\documentclass[10pt]{amsart}
\usepackage{amssymb}
\usepackage{color}
\usepackage{amsthm}
\usepackage{mathabx}

\newcommand{\llangle}{\langle\!\langle}
\newcommand{\rrangle}{\rangle\!\rangle}
\newcommand{\Arc}{\mathcal{A}}
\newcommand{\Curve}{\mathcal{K}}

\newcommand{\Diff}{\mathcal{D}}

\newcommand{\Diffmu}{\Diff_{\mu}}

\newcommand{\Func}{\mathcal{F}}

\newcommand{\RN}{\mathbb{R}^N}

\newcommand{\supnorm}[1]{\vvvert #1\vvvert}
\newcommand{\Ember}{\mathcal{I}}
\newcommand{\Imm}{\text{Imm}}
\DeclareMathOperator{\diver}{div}
\newcommand{\grad}{\nabla}

\newtheorem{theorem}{Theorem}[section]
\newtheorem{proposition}[theorem]{Proposition}
\newtheorem{lemma}[theorem]{Lemma}

\theoremstyle{definition}

\begin{document}

\title{The geometry of whips}
\author{Stephen C. Preston}
\address{Department of Mathematics, University of Colorado, Boulder,
CO 80309-0395} \email{Stephen.Preston@colorado.edu}

\maketitle

\tableofcontents

\section{Introduction}

The purpose of this paper is to explore the geometry of the
inextensible string (or whip) in Euclidean space $\RN$, and its application as an alternative geometry in shape recognition to the geometries proposed by Michor-Mumford~\cite{MM, MM2}, Younes et al.~\cite{younes}, Klassen et al.~\cite{KSMJ}, and others. Generally $N\ge 2$, although we will assume whenever convenient that $N=2$; this simplifies formulas but does not substantially change any of the results. 

In the absence of external forces, the
string is a geodesic motion in the space of unit-speed curves in $\mathbb{R}^N$; hence it is another of the examples of partial differential equations arising as geodesic motion on an infinite-dimensional manifold which have been discovered in the wake of Arnold's~\cite{arnold} approach to hydrodynamics (other examples include the Korteweg-deVries equation~\cite{ok}, the Camassa-Holm equation~\cite{mis}, and the Hunter-Saxton equation~\cite{KM}). The equations of motion are 
\begin{align}
\frac{\partial^2 \eta}{\partial t^2} &= \frac{\partial}{\partial s} \left( \sigma \, \frac{\partial \eta}{\partial s}\right), \label{stringevolution} \\
\frac{\partial^2\sigma}{\partial s^2} - \left \lvert \frac{\partial^2 \eta}{\partial s^2}\right\rvert^2 \sigma &= -\left\lvert \frac{\partial^2 \eta}{\partial s\partial t}\right\rvert^2,
\label{stringconstraint}
\end{align}
with initial conditions $\eta(0,s) = \gamma(s)$ and $\eta_t(0,s) = w(s)$, assumed to satisfy the compatibility conditions $\lvert \gamma'(s)\rvert \equiv 1$ and $\langle \gamma'(s), w'(s)\rangle \equiv 0$.
Here $t$ is the time parameter, $s$ is the length parameter along the curve, and $\sigma$ is the tension. We will suppose the curves have one fixed and one free endpoint. We will discuss other boundary conditions in Appendix \ref{circleappendix}.

As shown in the author's companion paper~\cite{whipsandchains} (where a slightly different notation was used to make the estimates more convenient), the easiest way to handle the fixed point is to extend the curve through the fixed point by oddness to get a curve with two free endpoints. So we have a curve $\eta\colon [0,T)\times [-1,1] \to \RN$ satisfying $\eta(-s)=-\eta(s)$, and the boundary conditions for \eqref{stringconstraint} are then $\sigma(-1)=\sigma(1)=0$. We then automatically have, for any sufficiently smooth solution, that $\sigma(-s)=\sigma(s)$,
so that all even derivatives of $\eta$ and all odd derivatives of $\sigma$ vanish at $s=0$. Using a spatial discretization (the method of lines), the author proved local existence and uniqueness in the weighted energy norm 
\begin{equation}\label{weightedenergy}
E_m = \sum_{j = 0}^m \int_{-1}^1 (1-s^2)^j \lvert \partial_s^j \eta_t\rvert^2 + (1-s^2)^{j+1} \lvert \partial_s^{j+1}\eta\rvert^2 \, ds
\end{equation}
for $m\ge 3$. Our aim in this paper is to explore the geometric interpretation of this result, and on the way we will obtain a result on the dependence of solutions on the initial conditions. Specifically, we show that for a fixed $\gamma$, the solution is differentiable but not $C^1$ as a function of $w$.

The system \eqref{stringevolution}--\eqref{stringconstraint} have been known and studied for hundreds of years, although only recently has there been a rigorous proof of well-posedness for the full nonlinear system (see \cite{whipsandchains}). Thess et al. \cite{thess} studied these equations on the circle as a toy model of hydrodynamical blowup, and one of our motivations in studying the geometrical aspects was to see just how far this analogy goes (comparing to Arnold's geometrical approach to hydrodynamics~\cite{arnold}). See also Serre~\cite{serre} and Reeken \cite{reeken, reeken1, reeken2} for analytical aspects of these equations, and references cited in \cite{whipsandchains} for physical discussions of their properties.

First, in Section \ref{submanifold} we define precisely the manifold structure on the space of curves and show that the inextensible curves form a smooth submanifold. Since we want to use the implicit function theorem to do this, we want to work on a Hilbert manifold, and the result \eqref{weightedenergy} suggests the appropriate Sobolev topology is the weighted space
\begin{equation}\label{curvespace}
\Curve^m = \{ \eta\colon [-1,1]\to \RN : R\eta=\eta \; \text{and}\; \lVert \eta\rVert_{j,j} < \infty \text{ for $0\le j\le m$}\} 
\end{equation}
where $R$ is the odd reflection operator $(R\eta)(s)=-\eta(-s)$ (which is bounded in any weighted Sobolev topology) and the weighted energy norm is defined by
\begin{equation}\label{weightednorm}
\lVert \eta\rVert^2_{j,k} =  \int_{-1}^1 (1-s^2)^j \lvert \partial_s^k\eta(s)\rvert^2 \, ds.
\end{equation}
We will first prove that the configuration space 
\begin{equation}\label{arcspace}
\Arc^m = \{ \eta\in \Curve^m : \lvert \eta'\rvert \equiv 1 \text{ and } R\eta=\eta\}
\end{equation}
is a smooth submanifold of $\Curve^m$ for $m\ge 4$, with tangent space given by 
\begin{equation}\label{arctangentspace}
T_{\gamma}\Arc^m = \{ v\in \Curve^m : \langle v', \gamma'\rangle \equiv 0 \text{ and } Rv=v\}.
\end{equation}
(The space $\Arc^m$ \emph{is} a smooth manifold as long as $m\ge 2$, but it fails to be a submanifold of $\Curve^m$ if $m=3$ or $m=4$.)


In Section \ref{geodesicsection} we define a weak Riemannian metric on $\Curve^m$ (and hence $\Arc^m$), given for vector fields $u$ and $v$ along a curve $\eta$, by 
\begin{equation}\label{curvesmetric}
\llangle u, v\rrangle_{\eta} = \int_{-1}^1 \langle u(s), v(s)\rangle \, ds.
\end{equation} 
Although we have a smooth metric
on a smooth manifold, the Levi-Civita connection is not smooth, and thus the geodesic
equation is not an ordinary differential equation. (This is typical behavior for a weak metric on an infinite-dimensional manifold; 
the Levi-Civita connection is unique if it exists, but it is not even guaranteed to exist if the Riemannian metric does not generate 
the topology of the manifold; see \cite{em}.) This is reflected in the fact that the right side of \eqref{stringevolution} is unbounded in any Sobolev topology. Hence we cannot get solutions of \eqref{stringevolution} by Picard iteration, and thus we are not guaranteed smooth dependence on initial conditions. 

In Section \ref{notc1section} we study the dependence of a solution $\eta(1,s)$ on the initial velocity field $w(s)$, given a fixed initial position $\gamma(s)$. The Riemannian exponential map on $\Arc^m$ is given by 
\begin{equation}\label{exponential}
\exp_{\gamma}(w) = \eta(1), \quad \text{where $\eta$ solves \eqref{stringevolution}--\eqref{stringconstraint}.}
\end{equation}
We will prove that  for $m\ge 3$, as long as $\gamma\in \Arc^{m+1}$, the exponential map is defined and continuous as a map from some open subset of $T_{\gamma}\Arc^m$ into $\Arc^m$. In fact we will show the exponential map is differentiable but not continuously differentiable. Our method is similar to that of Constantin-Kolev~\cite{ck} and Constantin-Kappeler-Kolev-Topalov~\cite{ckkt}: after establishing bounds on the linearized equation (to prove differentiability), we show that there are conjugate points arbitrarily close to $0$ by working out a very explicit special case (a string rotating like a rigid rod). If the exponential map were $C^1$, then the fact that $(d\exp_{\gamma})_0$ is the identity would imply by the inverse function theorem that there is a neighborhood of $0$ on which there are no conjugate points.

The failure of the exponential map to be $C^1$ has two consequences: one is that no geodesic can be minimizing, no matter how short (a conjugate point always implies the existence of a length-shortening variation); and the other is that we do not necessarily have geodesics joining two arcs (even nonminimizing geodesics), since the inverse function theorem is normally used to obtain this result. Hence the geometry of $\Arc^m$ experiences some genuinely infinite-dimensional phenomena. On the other hand, the distance function generated is nondegenerate, since $\Arc^m$ is a Riemannian submanifold of the geometrically flat space $\Curve^m$.

In Section \ref{curvaturesection} we compute the sectional curvature of $\Arc^m$, showing that it is always positive. Intuitively, this implies stability of geodesics by the Rauch comparison theorem~\cite{CE}; however, the fact that it is unbounded above implies that the rigorous study of stability via curvature estimates faces some technical difficulties. In fact even if we could apply the Rauch theorem, the presence of conjugate points arbitrarily close to the identity makes it impossible to get any rigorous information about the growth of Jacobi fields. 

Finally in Section \ref{michormumfordsection} we compare the geometry of $\Arc$ in the metric \eqref{curvesmetric} to other geometries on spaces of unparametrized curves, especially the $L^2$ metric on the space of parametrized curves modulo reparametrizations, studied by Michor-Mumford in \cite{MM}. Both metrics are too weak to preserve all the properties one expects in finite-dimensional geometry, but the metric \eqref{curvesmetric} has a nondegenerate distance while the Michor-Mumford metric gives a degenerate distance. We compare the geodesic equation \eqref{stringevolution}--\eqref{stringconstraint} with the geodesic equation on $\Arc$ obtained from the Michor-Mumford metric: the primary difference is that our metric is essentially a submanifold metric, while the Michor-Mumford metric is essentially a Riemmanian submersion metric on a homogeneous space. Since the submanifold metric \eqref{curvesmetric} is related to the physical $L^2$ metric and has a nondegenerate distance, we hope it may be of interest in shape recognition applications.

In Appendix \ref{circleappendix} we show how the results of this paper change if we consider periodic boundary conditions for the system \eqref{stringevolution}--\eqref{stringconstraint}. Many results actually become easier (for example, we can work in ordinary Sobolev spaces on the circle rather than weighted Sobolev spaces on the interval), and the essential features are the same. Then in Appendix \ref{nolengthappendix} we explore what happens if we remove the constraint that all our odd curves have length $1$; we see that many of the results break down in this case. 

Some of these results (in particular the nonnegativity of the sectional curvature in Theorem \ref{curvaturethm}) were first obtained by Victor Yudovich, but not to my knowledge published. Alexander Shnirelman introduced me to this problem, and I thank him for many useful discussions on it.

\section{The manifold structure of the arc space $\Arc^k$}\label{submanifold}

For this section we assume all curves map into $\mathbb{R}^2$, for simplicity. The space $\Curve^m$ of odd curves in $\mathbb{R}^2$ with the topology \eqref{curvespace} is obviously a manifold, as a linear space. 
The topology defined by the seminorms \eqref{curvespace} is sufficiently strong to make $\Arc^m$ a submanifold of $\Curve^m$ when $m\ge 4$, but when $m=2$ or $m=3$ the topology is almost but not quite strong enough. The difficulty here is that a bound on the norms $\lVert \eta\rVert_j$ for $0\le j\le 3$ is not sufficient to ensure boundedness of $\sup_{-1< s<1} \lvert \eta'(s)\rvert$, as shown by the example \begin{equation}\label{counterexample}
\eta'(s) = \arctan{\big(\ln{(1-s^2)}\big)}.
\end{equation}

First we recall the following lemma from \cite{whipsandchains}, relating the weighted Sobolev norms \eqref{weightednorm} and the weighted supremum norm
\begin{equation}\label{weightedsup}
\supnorm{f}^2_{j,k} = \sup_{-1\le s\le 1} (1-s^2)^j \lvert f^{(k)}(s)\rvert^2.
\end{equation}

\begin{lemma}\label{weightedestimateslemma}
For any real $j>0$ and any nonnegative integer $k$, and any smooth function $f$, we have the following estimates for the norms \eqref{weightednorm} and \eqref{weightedsup}:
\begin{align}
\lVert f\rVert^2_{j-1, k} &\lesssim \lVert f\rVert^2_{j,k} + \lVert f\rVert^2_{j+1, k+1}, \label{weightedpoincare} \\
\supnorm{f}^2_{j,k} &\lesssim \lVert f\rVert^2_{j,k} + \lVert f\rVert^2_{j+1, k+1}. \label{weightedsobolev}
\end{align}
\end{lemma}

\begin{theorem}\label{arcsubmfd}
If $m\ge 1$, then the space $\Arc^{m+1}$ defined by \eqref{arcspace} is a $C^{\infty}$ Hilbert manifold. If $m\ge 3$, then $\Arc^{m+1}$ is a $C^{\infty}$ Hilbert submanifold of the space $\Curve^{m+1}$ defined by \eqref{curvespace}.
\end{theorem}

\begin{proof}
If $\eta\in \Curve^2$, then we have $\int_{-1}^1 (1-s^2)^2 \lvert \eta''(s)\rvert^2 \, ds<\infty$. Hence by the standard Sobolev inequality on $[-1+\varepsilon, 1-\varepsilon]$ for any $\varepsilon>0$, we see that $\eta$ is in $C^1[-1+\varepsilon, 1-\varepsilon]$. Hence it makes sense to impose the condition $\lvert \eta'(s)\rvert \equiv 1$ for $s\in(-1,1)$, so that $\Arc^2$ is a closed subset of $\Curve^2$.

Furthermore for any such $\eta$, we can write $\eta'(s)=\big( \cos{\theta(s)}, \sin{\theta(s)}\big)$, where $\theta$ is uniquely determined once $\theta(0)$ is chosen. (Since $\eta'$ is continuous on $(-1,1)$, so is $\theta$.) We can easily compute that we have $\eta\in\Arc^2$ if and only if $\theta\in \Func^1[\mathbb{R}]$, where 
\begin{equation}\label{func}
\Func^m[\mathbb{F}] = \Big\{ f\colon [-1,1]\to \mathbb{\mathbb{F}} : \int_{-1}^1 (1-s^2)^{j+1} \lvert f^{(j)}(s)\rvert^2 \, ds < \infty \text{ for $0\le j\le m$}\Big\}
\end{equation}
for $\mathbb{F}=\mathbb{R}$ or $\mathbb{C}$.
Now $\Func^1$ is a Hilbert space, and the map $\eta\mapsto \theta$ defines coordinate charts (for example, on the set of $\eta$ with $\eta'(0)\ne v$ for any fixed unit vector $v$), for which the coordinate transition maps are trivially $C^{\infty}$. In this way we get a smooth Hilbert manifold structure on $\Arc^2$, and the same process will give a manifold structure on any $\Arc^{m+1}$: we just have to check that $\eta\in \Arc^{m+1}$ if and only if $\theta\in \Func^m[\mathbb{R}]$, which will follow from the next result.

To actually obtain $\Arc^{m+1}$ as a submanifold of $\Curve^{m+1}$, we can construct a coordinate chart on $\Curve^{m+1}$ which makes this obvious. We simply write $\eta'(s) = \big( e^{\psi(s)} \cos{\theta(s)}, e^{\psi(s)} \sin{\theta(s)}\big)$ for functions $\psi$ and $\theta$. Obviously if we think of $\mathbb{R}^2$ as $\mathbb{C}$, this is just the exponential $\eta'(s) = e^{\xi(s)}$ where $\xi(s)=\psi(s)+i\theta(s)$. Now our claim is that if $\psi$ is bounded (equivalently, if $\lvert \eta'\rvert$ is bounded both above and below away from zero), then $\eta\in\Curve^m$ if and only if $\xi\in\Func^{m-1}[\mathbb{C}]$. This is relatively easy to check using the F\`aa di Bruno formula for the derivative of a composition: we have $\eta'=e^{\xi}$ and $\xi=\ln{\eta'}$, so that for $m\ge 1$ the formula yields
\begin{align}
\frac{d^{m+1}\eta}{ds^{m+1}} &= e^{\xi} \sum \frac{m!}{k_1!\cdots k_m!} \prod_{j=1}^m \left( \frac{\xi^{(j)}(s)}{j!}\right)^{k_j}, \label{etatoxi} \\
\frac{d^m\xi}{ds^m} &= \sum \frac{m!}{k_1!\cdots k_m!} (-1)^{\ell-1} (\ell-1)! (\eta')^{\ell} \prod_{j=1}^m \left(\frac{\eta^{(j+1)}}{j!}\right)^{k_j}, \label{xitoeta}
\end{align}
where in the second sum we set $\ell = k_1+\cdots+k_m$, and where both sums are taken over all nonnegative integers $k_j$ such that $1\cdot k_1 + \cdots + m\cdot k_m = m$. 

We can now prove for $m\ge 1$ that $\xi\in \Func^m[\mathbb{C}]$ and $\sup {\Re \xi} <\infty$ implies $\eta\in \Curve^{m+1}$.
Ignoring the specific constants and using Cauchy-Schwarz, \eqref{etatoxi} yields
\begin{align*}
\lVert \eta\rVert^2_{m+1, m+1} &\lesssim \sup e^{2\Re\xi} 
\sum \int_{-1}^1 (1-s^2)^{m+1} 
\prod_{j=1}^m \lvert \xi^{(j)}(s)\rvert^{2k_j} \, ds \\
&\lesssim \sup e^{2\Re\xi} \sum \int_{-1}^1 (1-s^2) \prod_{j=1}^m \left( (1-s^2)^j \lvert \xi^{(j)}(s)\rvert^2\right)^{k_j} \, ds.
\end{align*}
Now if $j\le m-2$, we can pull out the term $\lvert (1-s^2)^j \lvert \xi^{(j)}\rvert^2$, using Lemma \ref{weightedestimateslemma} to get 
$$ \supnorm{\xi}^2_{j,j} \lesssim \lVert \xi\rVert^2_{j,j} + \lVert \xi\rVert^2_{j+1,j+1} \lesssim 
\lVert \xi\rVert^2_{j+1,j} + \lVert \xi\rVert^2_{j+2,j+1} + \lVert \xi\rVert^2_{j+3,j+2} < \infty$$ 
since $j+2\le m$ and $\xi\in \Func^m$. Hence we only need to worry about bounding
\begin{equation}\label{highestorder}
\int_{-1}^1 (1-s^2)^{m+1} \lvert \xi^{(m-1)}(s)\rvert^{2k_{m-1}} \lvert \xi^{(m)}(s)\rvert^{2k_m} \, ds.
\end{equation}
Since $(m-1)k_{m-1} + mk_m \le m$, we cannot have both $k_{m-1}$ and $k_m$ nonzero. It is then easy to check that in all possible cases, the term \eqref{highestorder} can be bounded in terms of the $\Func^m$ norm of $\xi$. 

We can similarly prove for $m\ge 1$ that $\eta\in \Curve^{m+1}$ and $\inf\lvert \eta'\rvert > 0$ implies $\xi\in \Func^m$, by using formula \eqref{xitoeta}. We note as a consequence that when $\xi$ is purely imaginary (corresponding to $\eta\in \Arc^m$), we get a bijective correspondence between an open subset of $\Arc^{m+1}$ and an open subset of $\{0\} \times \Func^m[\mathbb{R}]$ for any $m\ge 1$. This yields coordinate charts on $\Arc^{m+1}$ for $m\ge 1$.

Now to show that $\Arc^{m+1}$ is a submanifold of $\Curve^{m+1}$ for $m\ge 3$, we want to show that $m\ge 3$ implies that the coordinate chart $\eta'\to e^{\xi}$ gives a bijection between $\Curve^{m+1}$ and $\Func^m[\mathbb{C}]$, and the only thing remaining after the above estimates is getting upper and lower bounds on $\lvert \eta'\rvert$ and $\Re \xi$. Ideally we would use \eqref{weightedsobolev}, but that estimate fails since our case corresponds to $j=0$ (as mentioned in \eqref{counterexample}). Instead we use the usual (unweighted) Sobolev inequality on $[-1,1]$ to get
$$
\sup_{-1\le s\le 1} \lvert \psi(s)\rvert^2 \lesssim \int_{-1}^1 \lvert \psi(s)\rvert^2 + \int_{-1}^1 \lvert \psi'(s)\rvert^2 \, ds = \lVert \psi\rVert^2_{0,0} + \lVert \psi\rVert^2_{0,1},$$
and then apply \eqref{weightedpoincare} to get
\begin{equation}\label{supbound}
\begin{split}
\supnorm{\psi}^2_{0,0} &\lesssim \lVert \psi\rVert^2_{1,0} + \lVert \psi\rVert^2_{2,1} + \lVert \psi\rVert^2_{1,1} + \lVert \psi\rVert^2_{2,2} \\
&\lesssim \lVert \psi\rVert^2_{1,0} + \lVert \psi\rVert^2_{2,1} + \lVert\psi\rVert^2_{3,2} + \lVert \psi\rVert^2_{4,3}.
\end{split}
\end{equation}
Thus if $\psi=\Re \xi$, the fact that $\xi\in\Func^3[\mathbb{C}]$ implies $\psi$ is bounded. Similarly $\lvert \eta'\rvert$ can be bounded in terms of the norm of $\Curve^4$, so that the space of curves with $\lvert \eta'\rvert$ bounded away from zero is an open subset of $\Curve^4$. 

We conclude that for $m\ge 3$, the map $\eta\mapsto\xi$ is a coordinate chart on $\Curve^{m+1}$ which models $\Curve^{m+1}$ on the Hilbert space $\Func^m[\mathbb{C}]$. Furthermore locally the set $\Arc^{m+1}$ is mapped bijectively under this coordinate chart to the closed subspace $\{0\}\times \Func^m[\mathbb{R}]$, so that $\Arc^{m+1}$ is a Hilbert submanifold of $\Curve^{m+1}$.
\end{proof}

We note that on the way, we actually establish that $\Arc^{m+1}$ is a Banach submanifold of the slightly modified space $\overline{\Curve}^{m+1}$ for $m\ge 1 $, where the topology on $\overline{\Curve}^{m+1}$ is given by the Banach norm
$$ \sum_{j=0}^m \lVert \eta\rVert^2_{j,j} + \supnorm{\eta}^2_{0,1}.$$
For $m\ge 3$ the spaces $\Curve^{m+1}$ and $\overline{\Curve}^{m+1}$ are isomorphic since the extra term $\supnorm{\eta}^2_{0,1}$ becomes redundant due to \eqref{supbound}.

We can show directly that $\Arc^2$ is not a Hilbert submanifold of $\Curve^2$ (and similarly that $\Arc^3$ is not a Hilbert submanifold of $\Curve^3$) by the following calculation: we give $\Arc^2$ the coordinate chart $\Func^1$ and view $\Curve^2$ as a Hilbert space with the identity coordinate chart, and then the embedding of $\Arc^2$ into $\Curve^2$ is given in coordinates by the map $\Ember\colon \theta\mapsto (\cos{\theta}, \sin{\theta})$. The derivative of $\Ember$ at $\omega\in T_{\theta}\Func^1\cong \Func^1$ must be given by 
$$ \zeta \equiv D\Ember(\theta)(\omega) = \big( -\omega \sin{\theta}, \omega\cos{\theta}\big).$$ 
However if $\omega\in\Func^1$ and $\theta\in\Func^1$, we do not necessarily have $\zeta\in \Curve^2$, since 
$$ \lVert \zeta\rVert^2_{2,2} = \int_{-1}^1 (1-s^2)^2 \Big( \omega(s)^2 + \omega'(s)^2 + \omega(s)^2 \theta'(s)^2 \Big) \, ds,$$
and we need a bound on $\sup \omega$ in order for the last term to be bounded by $\lVert \theta\rVert^2_{2,1}$. But there is no reason any $\omega\in \Func^1$ has to have bounded supremum, due to examples like \eqref{counterexample}. Hence the embedding of $\Arc^2$ into $\Curve^2$ is not even differentiable, and so $\Arc^2$ cannot be a smooth submanifold of $\Curve^2$.

\section{The geodesic equation}\label{geodesicsection}

Since the space $\Curve^m$ defined by \eqref{curvespace} is a Hilbert space, it has an obvious Riemannian metric given by the Hilbert norm. Geodesics in this metric are always of the form $\eta(t) = \eta(0) + \eta'(0) t$, so the exponential map is $\exp_{\gamma}(w) = \gamma+tw$. This formula is the same regardless of whether we define the Riemannian metric by the weighted Sobolev norm \eqref{weightednorm}, as in
\begin{equation}\label{sobolevriemannian}
\llangle u, v\rrangle_{\gamma, m} = \int_{-1}^1 \sum_{j=0}^m (1-s^2)^j \langle D^ju(s), D^jv(s)\rangle \, ds, 
\end{equation}
or the weaker norm given by the kinetic energy,
\begin{equation}\label{kinetic}
\llangle u, v\rrangle_{\gamma} = \int_{-1}^1 \langle u(s), v(s)\rangle \, ds,
\end{equation}
which of course is \eqref{sobolevriemannian} when $m=0$.
However, geodesics on the submanifold $\Arc^m$ will be different depending on which choice we make. The natural choice from the perspective of physics is \eqref{kinetic}, which is a weak metric (it is not equivalent to the Hilbert norm). 

Formula \eqref{kinetic} obviously gives a smooth metric on the manifold $\Curve^m$, while as discussed in Section \ref{submanifold} the space $\Arc^m$ is a smooth submanifold of $\Curve^m$ if $m\ge 4$; nonetheless the connection on $\Arc^m$ is \emph{not} smooth. The reason for this is that the connection of a submanifold is obtained from the connection on the full manifold by orthogonal projection, and the orthogonal projection operator is not smooth (using the metric \eqref{kinetic}). In other situations when we get a smooth ODE on a submanifold, it is due to smoothness of the orthogonal projection: see for example \cite{em} and \cite{MP}.

\subsection{The orthogonal projection}

We will show in this subsection that the orthogonal projection is intimately related to the ODE \eqref{stringconstraint}. Since it will appear repeatedly, we summarize the main properties of the Green function for it, as proved in \cite{whipsandchains}.

\begin{theorem}\label{greenthm}
Let $G\colon [-1,1]\times [-1,1]\to \mathbb{R}$ denote the Green function of \eqref{stringconstraint}, satisfying
\begin{equation}\label{greenfunction}
\frac{\partial^2G}{\partial s^2} - \lvert \gamma''(s)\rvert^2 G(s,x) = -\delta(s-x), \qquad G(1,x)=G(-1,x)=0,
\end{equation}
so that the solution of \eqref{stringconstraint} is 
\begin{equation}\label{ODEsolgreen}
\sigma(s) = \int_{-1}^1 G(s,x) \lvert \eta_{tx}\rvert^2 \, dx.
\end{equation}
Then $G(s,x)=G(x,s)$ for all $s$ and $x$. In addition $G(s,x)\ge 0$ for all $s$ and $x$, and $G(s,x)=0$ only on the boundary.
Furthermore if $\overline{G}(s,x) = \frac{1}{2}\big[ G(s,x) + G(s,-x)\big]$ is the even symmetrized Green function, then \begin{equation}\label{greenupper}
\overline{G}(s,x) \le \min\{1-\lvert s\rvert, 1-\lvert x\rvert\}\end{equation}
and \begin{equation}\label{greenlower}
G(s,x) \ge (1-\lvert s\rvert)(1-\lvert x\rvert) e^{-\varrho}/(1+\varrho)\quad \text{where}\quad \varrho = \int_0^1 (1-\lvert s\rvert) \lvert \eta_{ss}\rvert^2 \,ds.
\end{equation}
\end{theorem}

\begin{proposition}\label{orthoprojprop}
Suppose $m\ge 2$. The orthogonal projection $P_{\gamma}\colon T_{\gamma}\Curve^m\to T_{\gamma}\Arc^m$ generated by the metric \eqref{kinetic} is given by 
\begin{align}
P_{\gamma}(z) &= z - (\sigma \gamma')' \quad \text{where $\sigma$ solves}\quad \label{orthoproj}\\
\sigma''(s) - \lvert \gamma''(s)\rvert^2 \sigma(s) &= \langle z'(s), \gamma'(s)\rangle \quad \text{with $\sigma(-1)=\sigma(1)=0$.}\label{sigmaeq}
\end{align}

For any fixed $\gamma\in \Arc^{m+2}$, the projection is continuous from $T_{\gamma}\Curve^m$ to $T_{\gamma}\Arc^m$, but the map is not continuous from $T\Curve^m$ to $T\Arc^m$.
\end{proposition}

\begin{proof}
We first observe that for any function $\sigma$ with $\sigma(1)=\sigma(-1)=0$, the vector field $v = \frac{d}{ds}(\sigma \gamma')$ is orthogonal to $T_{\gamma}\Arc^m$ in the metric \eqref{kinetic}, since if $\langle w', \gamma'\rangle \equiv 0$ then 
\begin{align*}
\llangle v, w\rrangle_{\gamma} &= \int_{-1}^1 \big\langle w(s), \tfrac{d}{ds}\big( \sigma(s)\gamma'(s)\big) \big\rangle\, ds \\
&= \big( \sigma(s) \langle w(s), \gamma'(s)\rangle\big)\big|_{s=-1}^{s=1} - \int_{-1}^1 \sigma(s) \langle w'(s), \gamma'(s)\rangle \, ds \\
&= 0.
\end{align*}

Now given $z\in T_{\gamma}\Curve^m$, solve the ordinary differential equation \eqref{sigmaeq} for $\sigma$, and define $w=z-(\sigma \gamma')'$. Then we have 
\begin{align*} 
\langle w'(s), \gamma'(s)\rangle &= \langle z'(s), \gamma'(s)\rangle - \langle \sigma''(s) \gamma'(s) + 2\sigma'(s)\gamma''(s) + \sigma(s)\gamma'''(s), \gamma'(s)\rangle \\
&= \langle z'(s), \gamma'(s)\rangle - \sigma''(s) + \sigma(s)\lvert \gamma''(s)\rvert^2 = 0,
\end{align*}
using the fact that $\lvert \gamma'(s)\rvert^2\equiv 1$ implies that $\langle \gamma'(s), \gamma''(s)\rangle\equiv 0$ and hence $\langle \gamma'(s), \gamma'''(s)\rangle = -\lvert \gamma''(s)\rvert^2$. So $w$ actually satisfies the tangent condition $\langle w', \gamma'\rangle\equiv 0$.

We just need to check that $w$ actually is in $T_{\gamma}\Arc^m$, i.e., that $\lVert w\rVert^2_{j, j}<\infty$ for $1\le j\le m$ as long as $z\in T_{\gamma}\Curve^m$ for $m\ge 4$. It is obviously sufficient (and easier) to check that $v\in T_{\gamma}\Curve^m$. Our estimates are based on the same estimates that are done in \cite{whipsandchains}, to which we refer for more details.

The key is that, by the product rule, $v=(\sigma \gamma')'$ satisfies
\begin{equation}\label{vestimate}
\begin{split}
\lVert v\rVert^2_{m,m} &\lesssim \sum_{k=0}^{m+1} \int_{-1}^1 (1-s^2)^m \lvert D^{m+1-k}\sigma\rvert^2 \lvert D^{k+1}\gamma\rvert^2 \, ds \\ 
&\lesssim A^2 \big( \lVert \gamma\rVert^2_{m+2, m+2} + \lVert \gamma\rVert^2_{m, m+1}\big) + \sum_{\ell=0}^{m-1} \supnorm{\gamma}^2_{m-\ell-1, m-\ell} \lVert \sigma\rVert^2_{\ell+1, \ell+2},
\end{split}
\end{equation}
where $A = \sup_s \lvert \sigma'\rvert \ge \sup_s \lvert \sigma\rvert/(1-s^2).$ Based on the bound \eqref{greenupper} for the Green function for \eqref{sigmaeq}, we easily see that 
$$ A^2 \lesssim \lVert z\rVert^2_{0,1} \lesssim \lVert z\rVert^2_{1,1} + \lVert z\rVert^2_{2,2},$$ using \eqref{weightedpoincare}. We easily get 
$$\supnorm{\gamma}^2_{m-\ell-1, m-\ell} \lesssim \lVert \gamma\rVert^2_{m-\ell, m-\ell} + \lVert \gamma\rVert^2_{m-\ell+1, m-\ell+1} + \lVert \gamma\rVert^2_{m-\ell+2, m-\ell+2}$$
using \eqref{weightedpoincare}--\eqref{weightedsobolev}, which allows us to bound $\lVert v\rVert^2_{m,m}$ in terms of the $\Arc^{m+2}$ norm of $\gamma$, once we get a bound on $\lVert \sigma\rVert^2_{\ell+1, \ell+2}$ for $0\le \ell\le m-1$. 

To obtain this, we use \eqref{sigmaeq} along with the product rule and Lemma \ref{weightedestimateslemma} to get the recursive inequality
$$ \lVert \sigma\rVert^2_{\ell+1, \ell+2} \lesssim \big( \textstyle \sum_{i=1}^{\ell+2} \lVert \gamma\rVert^2_{i,i} \big)^2 \big(\textstyle \sum_{j=2}^{\ell} \lVert \sigma\rVert^2_{j-1, j}+ \sum_{j=0}^{\ell+1} \lVert z\rVert^2_{j, j} \big).$$ Combining this with \eqref{vestimate} gives 
$$ \sum_{\ell=0}^m \lVert v\rVert^2_{m, m} \lesssim G_m\left( \sum_{i=1}^{m+2} \lVert \gamma\rVert^2_{i, i}\right) \left( \sum_{\ell=0}^m \lVert z\rVert^2_{\ell, \ell}\right)$$
for $m\ge 2$, for some function $G_m$. 

On the other hand, if $\gamma\in \Arc^m$ and not $\Arc^{m+1}$, we obviously do not in general have $(\sigma \gamma')' \in T_{\gamma}\Curve^m$ even if $\sigma$ is $C^{\infty}$. Hence even if $z$ is $C^{\infty}$, the projection $P_{\gamma}(z)$ given by \eqref{orthoproj} is not in $T_{\gamma}\Arc^m$ if $\gamma$ is only in $\Arc^m$.
\end{proof}

\subsection{The second fundamental form}

The orthogonal projection encodes all the geometry of the submanifold $\Arc^m$, via the second fundamental form. The second fundamental form then leads to both the geodesic equation and to the sectional curvature (which we will discuss in Section \ref{curvaturesection}).

\begin{lemma}\label{secondfundy}
The second fundamental form of $\Arc^m$, as a submanifold of $\Curve^m$, is given by the operator $S\colon T_{\gamma}\Arc^m\times T_{\gamma}\Arc^m \to (T_{\gamma}\Arc^m)^{\perp}$ defined by 
\begin{equation}\label{secondfundamental}
S(u,v) = \frac{d}{ds}\big(\sigma_{uv}(s)\gamma'(s)\big), \quad \text{where}
\end{equation}
\begin{equation}\label{sigmauv}
\sigma_{uv}(s) = \int_{-1}^1 G(s,x) \langle u'(x), v'(x)\rangle \, dx
\end{equation}
in terms of the Green function of Theorem \ref{greenthm}.
This operator is only well-defined if $\gamma\in \Arc^{m+2}$. 
\end{lemma}

\begin{proof}
In general the second fundamental form is the orthogonal projection of the connection: if $U$ and $V$ are vector fields on $\Arc^m$, with $u=U_{\gamma}$ and $v=V_{\gamma}$ the values of these fields at $\gamma\in\Arc^m$, then the second fundamental form is $S(u,v) = \big((\nabla_UV)_{\gamma}\big)^{\perp}$, and the value obtained depends only on the values $U_{\gamma}$ and $V_{\gamma}$ (not on the extensions $U$ and $V$). (See for example do Carmo~\cite{docarmo2}.) Unfortunately it is somewhat awkward to work with general vector fields on a function space, and so we use the alternative method of vector fields along curves. 

So suppose $\eta(t)$ is a curve in $\Arc^m$ with $\eta(0)=\gamma$, and let $V(t)$ be a curve along $\eta$, so that $V(t)\in T_{\eta(t)}\Arc^m$ for each $t$. Set $u=\dot{\eta}(0)$ and $v=V(0)$. Then the covariant derivative of $V$ in the direction $u$, calculated in the flat ambient manifold $\Curve^m$, is $$(\nabla_uV)_{\gamma} = \frac{DV}{dt}(0) = \frac{dV}{dt}(0)$$
where the last equality comes from using flatness of the Hilbert manifold $\Curve^m$ to change the covariant derivative to an ordinary derivative. Now using formula \eqref{orthoproj} for the orthogonal projection, we get 
$$ \sigma_{uv}''(s)-\lvert \gamma''(s)\rvert^2 = \langle V_{ts}, \eta_s\rangle_{t=0}.$$
To simplify $\langle V_{ts}, \eta_s\rangle$, we use the fact that $V(t)\in T_{\eta(t)}\Arc^m$ at each time to obtain $\langle V_s, \eta_s\rangle \equiv 0$ for all time, so that differentiating we get $\langle V_{st}, \eta_s\rangle + \langle V_s, \eta_{st}\rangle = 0$. Now at time $0$ we have $V_s=v'$ and $\eta_{st} = u'$, so that $\langle V_{ts}, \eta_s\rangle_{t=0} = -\langle u', v'\rangle$, which implies 
$$\sigma_{uv}''(s) - \lvert \gamma''(s)\rvert^2 \sigma_{uv}(s) = -\langle u'(s), v'(s)\rangle.$$
Hence the formula \eqref{ODEsolgreen} yields \eqref{sigmauv}.
\end{proof}

As a consequence we obtain the geodesic equation on $\Arc^m$, using the general formula for a geodesic on a submanifold:
$$ \frac{D^2\eta}{dt^2} = S\left( \frac{d\eta}{dt}, \frac{d\eta}{dt}\right),$$ 
which using \eqref{secondfundamental} reduces to \eqref{stringevolution} with tension given by \eqref{stringconstraint}.

The fact that the orthogonal projection fails to be continuous in both $\gamma$ and $z$ implies that, unlike in Ebin-Marsden~\cite{em}, the weak geodesic equation on $\Arc^m$ cannot be viewed as an ordinary differential equation, and hence cannot be solved via Picard iteration.

On the other hand, we can prove local existence and uniqueness of solutions. The main result of the author's companion paper~\cite{whipsandchains} is the following theorem (restated here more geometrically):
\begin{theorem}\label{maincompanion}
Suppose $m\ge 3$. If $\gamma\in \Arc^{m+1}$ and $w\in T_{\gamma}\Arc^m$, then there is a $T>0$ such that there is a unique solution $\eta\colon (-T,T)\to \Arc^{m+1}$ of \eqref{stringevolution}--\eqref{stringconstraint} satisfying $\eta(0)=\gamma$, $\dot{\eta}(0)=w$, and such that $\eta(t)\in \Arc^{m+1}$ and $\dot{\eta}(t)\in T_{\eta(t)}\Arc^m$ for all $t$. 
\end{theorem}

The loss of derivatives here (i.e., the fact that $\dot{\eta}$ is not as smooth as $\eta$) means that we do not have a classical exponential map (which would be a map from an open subset of $T\Arc^m$ to itself). 

However, if we \emph{fix} an initial configuration $\gamma\in \Arc^{m+1}$, then we have a reduced exponential map 
\begin{equation}\label{reducedexponential} 
\exp_{\gamma}\colon \Omega\subset T_{\gamma}\Arc^m \to \Arc^m
\end{equation}
defined on some open neighborhood $\Omega$ of $0$ by $\exp_{\gamma}(w) = \eta(1)$, where $\eta$ solves \eqref{stringevolution}--\eqref{stringconstraint} with $\eta(0)=\gamma$ and $\dot{\eta}(0)=w$. Actually as mentioned $\exp_{\gamma}(w)$ is really in $\Arc^{m+1}$, but the theorem in the next section on continuous dependence can only be proved in this weaker topology.

\section{Differentiability of the reduced exponential map}\label{notc1section}

We now want to establish continuity of the reduced exponential map \eqref{reducedexponential}; in other words, for a fixed, sufficiently smooth initial whip configuration $\gamma$, we show that the solution depends continuously on the initial velocity $w$ in any weighted Sobolev topology \eqref{weightednorm}. In fact we will obtain Lipschitz continuity as a result of showing that $\exp_{\gamma}$ is differentiable everywhere on $T_{\gamma}\Arc^m$, with bounded derivative, but that it is not $C^1$. The latter fact is a consequence of clustering of conjugate points near $0$, the same thing that happens for the exponential maps corresponding to Burgers' equation~\cite{ck} and the Korteweg-deVries equation~\cite{ckkt}.

First we compute the derivative of the exponential map (which is just the linearization of the equations \eqref{stringevolution}--\eqref{stringconstraint}) and prove that it is bounded.

\begin{theorem}\label{differentiable}
Suppose $m\ge 3$ and $\gamma\in \Arc^{m+1}$, and let $w\in T_{\gamma}\Arc^m$. 

The derivative of $\exp_{\gamma}\colon T_{\gamma}\Arc^m\to \Arc^m$ is 
$D(\exp_{\gamma})_w(y) = \xi(1)$ where $(\xi, \psi)$ is a solution of the linearized equations 
\begin{align}
\xi_{tt} &= \partial_s(\sigma \xi_s) + \partial_s(\phi \eta_s), \label{linearizedevolution} \\ 
\phi_{ss} - \lvert \eta_{ss}\rvert^2 \phi &= 2\langle \eta_{ss}, \xi_{ss}\rangle \sigma - 2\langle \eta_{st}, \xi_{st}\rangle \label{linearizedconstraint}
\end{align}
where $\langle \xi_s, \eta_s\rangle \equiv 0$, with initial conditions $\xi(0,s) = 0$ and $\xi_t(0,s) = y(s)$. Here $\eta$ and $\sigma$ solve \eqref{stringevolution}--\eqref{stringconstraint} with initial conditions $\eta(0,s) = \gamma(s)$ and $\eta_t(0,s) = w(s)$. 

The derivative satisfies the bound 
\begin{equation}\label{derivbound}
\sum_{k=0}^m \lVert D(\exp_{\gamma})_w(y)\rVert^2_{k, k} \lesssim \textstyle H_m\left(\sum_{i=2}^{m+1} \lVert \gamma\rVert^2_{i, i}, \sum_{j=1}^m \lVert w\rVert^2_{j, j}\right) \sum_{k=0}^m \lVert y\rVert^2_{k, k}
\end{equation}
for some function $H_m$.
\end{theorem}

\begin{proof}
We obtain \eqref{linearizedevolution}--\eqref{linearizedconstraint} by considering a family of solutions $\big( \eta(r, t, s), \sigma(r, t, s)\big)$ depending on a parameter $r$, satisfying $\eta(r, 0, s) = \gamma(s)$, $\eta_t(0, 0, s) = w(s)$, and $\eta_{tr}(0, 0, s)=y(s)$. Setting $\xi = \eta_r\big|_{r=0}$ and $\phi=\sigma_r\big|_{r=0}$, we get the desired equations and initial conditions. 

The bound \eqref{derivbound} will be obtained by bounding the tension-weighted energy norms
\begin{equation}\label{epsilontilde}
\tilde{\varepsilon}_{m-1} = \sum_{\ell=0}^{m-1} \int_{-1}^1 \Big( \sigma(t,s)^{\ell} \lvert \partial_s^{\ell}\partial_t \xi(t,s)\rvert^2 + \sigma(t,s)^{\ell+1} \lvert \partial_s^{\ell+1}\xi(t,s)\rvert^2 \Big) \, ds.
\end{equation}
As in the estimates of \cite{whipsandchains}, we compare the tension-weighted norm to the time-independent weighted energy norm
\begin{equation}\label{epsilon}
\varepsilon_{m-1} = \sum_{\ell=0}^{m-1} \int_{-1}^1 \Big( s^{\ell} \lvert \partial_s^{\ell}\partial_t \xi(t,s)\rvert^2 + s^{\ell+1} \lvert \partial_s^{\ell+1} \lvert \partial_s^{\ell+1}\xi(t,s)\rvert^2 \Big) \, ds.
\end{equation}
The bounds from Theorem \ref{greenthm} imply, as in \cite{whipsandchains}, that we have bounds $A(t)$ and $B(t)$ such that 
$$ 0< \frac{1}{B(t)} \le \frac{\sigma(t,s)}{1-s^2} \le A(t) < \infty$$
implies that the norms \eqref{epsilontilde} and \eqref{epsilon} are equivalent. We recall also the result from \cite{whipsandchains} that $\sup_s \lvert \sigma_t(t,s)\rvert/(1-s^2) \le \sup_s \lvert \sigma_{st}\rvert = C(t)$ is bounded in terms of the $\Arc^4$ norm of $\gamma$ and the $\Arc^3$ norm of $w$. Finally we write 
$$ E_m(t) = \sum_{\ell=0}^m \int_{-1}^1 \Big( (1-s^2)^{\ell} \lvert \partial_s^{\ell}\eta_t\rvert^2 + (1-s^2)^{\ell+1} \lvert \partial_s^{\ell+1} \eta\rvert^2 \Big) \, ds.$$ By our assumption on the initial conditions and Theorem \ref{maincompanion}, $E_m(t)$ is bounded.

Now we compute the time derivative $d\tilde{\varepsilon}_{m-1}/dt$ one term at a time, obtaining after some simplifications
\begin{align*}
&\frac{d}{dt} \int_{-1}^1 \sigma^{\ell} \lvert \partial_s^{\ell}\partial_t \xi\rvert^2 + \sigma^{\ell+1} \lvert \partial_s^{\ell+1}\xi \rvert^2 \, ds \lesssim A^{\ell-1} C \lVert \xi_t\rVert^2_{\ell, \ell} + 
A^{\ell}C \lVert \xi\rVert^2_{\ell+1,\ell+1} \\
&\qquad\qquad + A^{\ell} \lVert \xi_t\rVert_{\ell, \ell} \sum_{j=2}^{\ell+1} \sqrt{\int_{-1}^1 (1-s^2)^{\ell} \lvert \partial_s^j\sigma\rvert^2 \lvert \partial_s^{\ell+2-j}\xi\rvert^2 \, ds} \\
&\qquad\qquad + A^{\ell} \lVert \xi_t\rVert_{\ell, \ell} \sum_{j=2}^{\ell+1} \sqrt{\int_{-1}^1 (1-s^2)^{\ell} \lvert \partial_s^j\phi\rvert^2 \lvert \partial_s^{\ell+2-j}\eta\rvert^2 \, ds}   \\
&\qquad\qquad  + A^{\ell} \Delta \lVert \xi_t\rVert_{\ell, \ell} \big(\lVert \eta\rVert_{\ell+2, \ell+2} + \lVert \eta\rVert_{\ell, \ell+1}\big) \\
&\qquad\qquad + \int_{-1}^1 \partial_s\big( \sigma^{\ell+1} \langle \partial_s^{\ell+1}\xi, \partial_s^{\ell}\xi_t\rangle\big) \, ds,
\end{align*}
where $\Delta=\sup_s \lvert \phi_s\rvert$. 

The terms on the first line are obviously bounded by $\varepsilon_{m-1}$, and the terms on the last line vanish because of the boundary condition on $\sigma$. The fourth line is bounded in terms of $E_m$ and $\varepsilon_{m-1}$, using a bound on $\Delta$ obtained by using the Green function bound \eqref{greenupper}, together with the weighted Sobolev bound \eqref{weightedpoincare}, to get 
$$ \Delta^2 \lesssim A^2 \int_{-1}^1 (1-s^2) \lvert \xi_{ss}\rvert^2 \lvert \eta_{ss}\rvert^2 \, ds + \int_{-1}^1 \lvert \xi_{st}\rvert^2 \lvert \eta_{st}\rvert^2 \,ds \lesssim (A^2+1) \varepsilon_2 E_2.$$ The terms on the second line can be bounded using the fact from \cite{whipsandchains} that $\lVert\sigma\rVert^2_{j+1,j+1}$ can be bounded in terms of $E_j$ for $j\ge 3$. Finally the terms on the third line can be bounded by obtaining a bound on $\lVert \phi\rVert^2_{j-1, j}$ in terms of $\varepsilon_{j-1}$, which is obtained in the same way as the proof of Proposition \ref{orthoprojprop}. 

Summing from $\ell=0$ to $m-1$, we obtain an estimate of the form
$$ \frac{d\tilde{\varepsilon}_{m-1}}{dt} \le J(E_m) \varepsilon_{m-1}$$
for some function $J$, and then the bound $1-s^2 \le B\sigma$ gives us a Gronwall inequality of the form $d\tilde{\varepsilon}_{m-1}/dt \le \tilde{J}(E_m) \tilde{\varepsilon}_{m-1}$. We then obtain the desired bound \eqref{derivbound} from this.
\end{proof}

Integrating the derivative obviously gives us a bound on $\lVert \exp_{\gamma}(w)-\exp_{\gamma}(v)\rVert_{m,m}$, as in Cheeger-Ebin~\cite{CE}, which establishes that the reduced exponential map \eqref{reducedexponential} is locally Lipschitz, as desired.

However, the exponential map cannot be $C^1$; if it were, then the fact that its differential is invertible at zero would imply it is also invertible near zero. But the latter does not happen.

\begin{theorem}\label{notc1}
The reduced exponential map \eqref{reducedexponential} is not a $C^1$ map on $T_{\gamma}\Arc^m$ for any $m$. 
\end{theorem}

\begin{proof}
For any $\gamma$, the differential $(D\exp_{\gamma})_w$ at $w=0$ is the identity, which is easy to see from the fact that in this case, $\eta(t)=\gamma$ and $\sigma(t)=0$ for all $t$ in \eqref{stringevolution}--\eqref{stringconstraint}. This implies that the solution $\phi$ of \eqref{linearizedconstraint} is zero, so that \eqref{linearizedevolution} reduces to $\xi_{tt}=0$. 
Since $\xi(0)=0$ and $\dot{\xi}(0)=y$, we get $\xi(1)=y$, i.e., $(D\exp_{\gamma})_0(y)=y$. If the exponential map were $C^1$, then $(D\exp_{\gamma})_w$ would have to also be invertible for sufficiently small $w$.

However, we can find $w$ arbitrarily close to $0$ in $T_{\gamma}\Arc^m$ for which $(D\exp_{\gamma})_w$ is not an isomorphism. To do this, we work out an explicit solution in detail. It is easy to verify that 
\begin{equation}\label{simplesoln}
\eta(t,s) = (s\cos{\omega t}, s\sin{\omega t}) \quad \text{and} \quad \sigma(t, s) = \tfrac{\omega^2}{2} (1-s^2)
\end{equation}
form a solution of \eqref{stringevolution}--\eqref{stringconstraint} for any angular velocity $\omega$. In this case we have of course \begin{equation}\label{gammawsimple}
\gamma(s) = (s, 0) \quad \text{and}\quad w(s) = (0, \omega s).
\end{equation} 

The constraint $\langle \xi_s, \eta_s\rangle \equiv 0$ implies that $\xi_s(t,s) = \chi(t,s) (-\sin{\omega t}, \cos{\omega t})$ for some function $\chi$. Differentiating \eqref{linearizedevolution} with respect to $s$ gives $\phi_{ss}\equiv 0$ and 
\begin{equation}\label{xieq}
\chi_{tt} - \omega^2 \chi = \tfrac{\omega^2}{2} \partial_s^2((1-s^2) \chi).
\end{equation}
Since $\xi$ is odd as a function of $s$, we must have $\chi$ even as a function of $s$. Expanding $\chi$ in a basis of derivatives of odd Legendre polynomials as 
$$ \chi(t, s) = \sum_{n=1}^{\infty} \chi_n(t) P_{2n-1}'(s)$$
and using the fact that 
$$ \frac{d^2}{ds^2} \Big( (1-s^2) P_{2n-1}'(s)\Big) = -2n(2n-1) P_{2n-1}'(s),$$
we see that 
$$ \chi_n''(t) - \omega^2 \chi_n(t) = -\omega^2 n(2n-1) \chi_n(t),$$
the solution of which, with $\chi_n(0)=0$ and $\chi_n'(0)=c_n$, is 
$$ \chi_n(t) = \frac{c_n}{\alpha_n} \sin{(\alpha_n t)},$$
where $\alpha_n^2 = \omega^2 (2n+1)(n-1)$. 

So if $y_n(s)=(0, P_{2n-1}'(s))$ for any $n$, then we have $(D\exp_{\gamma})_w(y_n) = \frac{\sin{\alpha_n}}{\alpha_n} y_n$. Now for each $n\ge 2$ we can define $\omega_n = \sqrt{\pi/((2n+1)(n-1))}$, so that $\alpha_n=\pi$. Obviously $\omega_n\to 0$ and so the corresponding $w_n$ in \eqref{gammawsimple} converge to $0$ in $T_{\gamma}\Arc^m$ for any $m$, yet $(D\exp_{\gamma})_{w_n}$ has a nontrivial kernel for any $n\ge 2$. Hence even for a fixed smooth $\gamma$, the differential $w\mapsto (D\exp_{\gamma})_w$ cannot be continuous as a map from $T_{\gamma}\Arc^m$ to $L(T_{\gamma}\Arc^m, T\Arc^m)$. 
\end{proof}

Generally speaking, if $\eta\colon [0,T]\to M$ is a geodesic in a finite-dimensional manifold $M$ such that 
$\eta(b)$ is conjugate to $\eta(0)$ for some $b\le T$, then $\eta$ cannot be locally minimizing on $[0,T]$. (See for example do Carmo~\cite{docarmo2}.) For the geodesic \eqref{simplesoln}, our computation shows that no matter how small $T$ is, there is a conjugate point at some $b<T$, so that even an arbitrarily short geodesic cannot be minimizing.

On the other hand, the induced distance is \emph{not} degenerate: if $\gamma_1$ and $\gamma_2$ are distinct curves in $\Arc$, then the infimum of lengths of paths in $\Arc$ joining them has a positive lower bound. This shows that the vanishing geodesic distance is not a consequence of having unbounded curvature, which was suggested in \cite{MM3}.

\begin{proposition}\label{nondegeneratedistance}
The Riemannian distance between distinct curves $\gamma_1$ and $\gamma_2$ in $\Arc$, in the metric 
\eqref{curvesmetric}, is always positive.
\end{proposition}

\begin{proof}
The idea is basically that the Riemannian distance in a submanifold of a flat space is always at least as large as the ``chord'' distance in the flat space. Specifically, if $\eta(t, s)$ is a curve with $\lvert \eta_s\rvert\equiv 1$ and $\eta(0,s)=\gamma_1(s)$ and $\eta(1,s)=\gamma_2(s)$, then the length of $\eta$ is bounded using the Cauchy-Schwarz inequality by
\begin{align*}
L(\eta) &= \int_0^1 \sqrt{\int_{-1}^1 \lvert \eta_t(t,s)\rvert^2 \, ds}\,dt 
\ge \int_0^1 \int_{-1}^1 \tfrac{1}{\sqrt{2}} \lvert \eta_t(t, s)\rvert \, ds \, dt \\
&\ge \frac{1}{\sqrt{2}} \int_{-1}^1 \left\lvert \int_0^1 \eta_t(t,s)\, dt\right\rvert \, ds
= \frac{1}{\sqrt{2}} \int_{-1}^1 \lvert \gamma_2(s)-\gamma_1(s)\rvert \, ds.
\end{align*}
So we get an absolute positive lower bound regardless of $\eta$.
\end{proof}

\section{Curvature of the arc space}\label{curvaturesection}

Having computed the second fundamental form of $\Arc^m$ in $\Curve^m$ in
Lemma \ref{secondfundy}, we can now compute the sectional curvature using
the Gauss-Codazzi formula.

\begin{theorem}\label{curvaturethm}
Let $m\ge 2$ and $\gamma\in\Arc^{m+2}$, and let $u, v\in T_{\gamma}\Arc^m$ be vector fields.
Then the sectional curvature $K(u,v)$ at $\gamma$ in the section spanned by $u$ and $v$ is
\begin{equation}\label{sectional}
K = \frac{\int_{-1}^1 \int_{-1}^1 G(s,x) \big( \lvert u'(s)\rvert^2 \lvert v'(x)\rvert^2 - \langle u'(x), v'(x)\rangle \langle u'(s), v'(s)\rangle 
\, ds\,dx}{\int_{-1}^1 \int_{-1}^1 \lvert u(s)\rvert^2 \lvert v(x)\rvert^2 - \langle u(x), v(x)\rangle \langle u(s), v(s)\rangle 
\, ds\, dx},
\end{equation}
where $G$ is the Green function \eqref{greenfunction}. 
The curvature satisfies $K\ge e^{-\varrho}/(1+\varrho)$, where $\varrho = \int_0^1 (1-s) \lvert \gamma''(s)\rvert^2 \, ds$; in particular it is always positive but never bounded above.
\end{theorem}

\begin{proof}
First we note that $\Curve^m$ is flat, so the Gauss-Codazzi formula gives 
$$ K = \frac{\llangle S(u,u), S(v,v)\rrangle - \llangle S(u,v), S(u,v)\rrangle}{\llangle u,u\rrangle \llangle v,v\rrangle - \llangle u,v\rrangle^2},$$
where $S$ is the second fundamental form \eqref{secondfundamental}. Now we have 
\begin{align*}
\llangle S(u,u), S(v,v)\rrangle &= \int_{-1}^1 \langle \sigma_{uu}'\gamma' + \sigma_{uu} \gamma'', 
\sigma_{vv}'\gamma' + \sigma_{vv}\gamma''\rangle \, ds \\
&= \int_{-1}^1 \sigma_{uu}' \sigma_{vv}' + \lvert \gamma''\rvert^2 \sigma_{uu}\sigma_{vv} \, ds 
= \int_{-1}^1 \sigma_{uu}(s) \lvert v'(s)\rvert^2 \, ds.
\end{align*}
Now by \eqref{sigmauv} we have 
$$ \sigma_{uu}(s) = \int_{-1}^1 G(s,x) \lvert u'(x)\rvert^2 \, dx,$$
where $G$ is the Green function \eqref{greenfunction}. Hence we can write 
$$ \llangle S(u,u), S(v,v)\rrangle = \int_{-1}^1 \int_{-1}^1 G(s,x) \lvert u'(x)\rvert^2 \lvert v'(s)\rvert^2 \, ds \, dx.$$
Formula \eqref{sectional} follows.

Nonnegativity of the sectional curvature follows from the fact that, by symmetry of the Green function,
we can write the numerator of \eqref{sectional} as 
\begin{multline*}
\frac{1}{2} \int_{-1}^1 \int_{-1}^1 G(s,x) M(s,x)\,ds\,dx,
\quad \text{where}\\
M(s,x) = \lvert u'(s)\rvert^2 \lvert v'(x)\rvert^2 + \lvert u'(x)\rvert^2 \lvert v'(s)\rvert^2 
- 2\langle u'(s), v'(s)\rangle \langle u'(x), v'(x)\rangle.
\end{multline*}
We have $G(s,x)\ge 0$ for all $s$ and $x$ by Theorem \ref{greenthm} and  $M(s,x)\ge 0$ for all $s$ and $x$ by the Cauchy-Schwarz inequality.

To get a sharper estimate, note that since $\gamma\in \Arc^4$ by assumption, we know $\lVert \gamma\rVert_{1,2}$ is bounded by \eqref{weightedpoincare}. Hence we can use the estimate \eqref{greenlower} to obtain 
$$ G(s,x) \ge \frac{e^{-\varrho}}{1+\varrho} (1-\lvert s\rvert)(1-\lvert x\rvert) \ge \frac{e^{-\varrho}}{4(1+\varrho)} (1-s^2) (1-x^2)$$
on $[-1,1]\times[-1,1]$, which allows us to write \eqref{sectional} as 
\begin{equation}\label{sectionallower}
K\ge \frac{e^{-\varrho}}{4(1+\varrho)} \, \frac{ \lVert u\rVert^2_{1,1} \lVert v\rVert^2_{1,1} - \llangle u, v\rrangle_{1,1}^2}{\lVert u\rVert^2_{0,0} \lVert v\rVert^2_{0,0} - \llangle u, v\rrangle^2_{0,0}}.
\end{equation} 
To get a lower bound on this, we disregard the restrictions on $u$ and $v$ (that they be elements of $T_{\gamma}\Arc^m$) and minimize over all odd vector fields along $\gamma$. Expand $u$ and $v$ in a basis of odd Legendre polynomials as $u(s) = \sum_{n=1}^{\infty} u_n P_{2n-1}(s)$ and $v(s) = \sum_{n=1}^{\infty} v_n P_{2n-1}(s)$, where $u_n$ and $v_n$ are vectors in $\RN$. Then the bound \eqref{sectionallower} becomes
$$ K \ge \frac{e^{-\varrho}}{4(1+\varrho)} \, \frac{ \sum_{n,m=1}^{\infty} \lambda_n \lambda_m \big( \lvert u_n\rvert^2 \lvert v_m\rvert^2 - \langle u_n, v_n\rangle \langle u_m, v_m\rangle\big)}{ \sum_{n,m=1}^{\infty} \big( \lvert u_n\rvert^2 \lvert v_m\rvert^2 - \langle u_n, v_n\rangle \langle u_m, v_m\rangle\big)} \ge  \frac{\lambda_1^2 e^{-\varrho}}{4(1+\varrho)},$$
where $\lambda_n = 2n(2n-1)$. Positivity of the curvature follows. It is easy to see that the curvature can be made arbitrarily large using this formula as well. 
\end{proof}

The fact that the curvature is unbounded above is responsible for the fact (as shown in Theorem \ref{notc1}) that conjugate points along a geodesic occur at times arbitrarily close to $0$, and hence for the failure of the Riemannian exponential map to be $C^1$: If the curvature were bounded above, then the Rauch comparison theorem would imply that there is a small interval of any geodesic in which no conjugate points can occur, contradicting Theorem \ref{notc1}.

\section{Comparison with other metrics}\label{michormumfordsection}

The space of curves is of interest in shape-recognition applications~\cite{MM}, since the first step in distinguishing two shapes in the plane is to distinguish their boundary curves. Obviously in studying 
geometry of curves for this purpose, we want to consider the image of the curve in the plane (which is all the viewer can see), not the actual map from the interval to the plane. There are essentially two ways to get a Riemannian structure on this set: impose a specific parametrization with unit speed (as we have done in this article so far), or consider all parametrized curves and mod out by the reparametrizations (the diffeomorphism group of the interval). The latter approach is the one taken by Michor and Mumford~\cite{MM}. However the approaches are basically equivalent:\footnote{Here we modify the Michor-Mumford space to consider curves with a free boundary and a fixed length; in Appendix \ref{circleappendix} we consider the periodic case, and in Appendix \ref{nolengthappendix} we discuss the case where length is not constrained.} an odd curve of length $2$ in the plane always has exactly one parametrization on $[-1,1]$ of unit speed, so that if \begin{equation}\label{Imm2}
\Imm_2([-1,1], \mathbb{R}^2) = \left\{ \eta\colon [-1,1]\to\mathbb{R}^2 : \lvert \eta'(s)\rvert \ne 0,\; \eta(-s)=-\eta(s),\; L(\eta)=2\right\}
\end{equation} denotes the space of odd immersions into $\mathbb{R}^2$ for which the image has length $2$ and $\Diff([-1,1])$ is the group of odd orientation-preserving diffeomorphisms of $[-1,1]$ to itself, then we expect to have $\Imm_2([-1,1])/\Diff([-1,1]) \cong \Arc$. This doesn't quite work rigorously since the action is not always free, but we can still see what the Michor-Mumford metric looks like on $\Arc$. In this section we will assume all objects are $C^{\infty}$ and work formally, although with a bit more work we could extend the results to weighted Sobolev spaces.

The reparametrization action of $\Diff([-1,1])$ on $\Imm_2([-1,1])$ is given by composition: for $h\in \Diff([-1,1])$ the map is $R_h(\eta)=\eta\circ h$. We can define a Riemannian metric on $\Imm_2$ by 
\begin{equation}\label{rightinvariantmetric}
\llangle u, v\rrangle_{\eta} = \int_{-1}^1 \langle u(s), v(s)\rangle \lvert \eta'(s)\rvert \, ds.
\end{equation}
We clearly have $\llangle u\circ h, v\circ h\rrangle_{\eta\circ h} = \llangle u, v\rrangle_{\eta}$ for any $h\in \Diff([-1,1])$ by the change of variables formula, so that the metric \eqref{rightinvariantmetric} is invariant under the group action. This is in contrast to the metric \eqref{curvesmetric}, which is \emph{not} invariant under reparametrizations. (Of course on the submanifold of unit-speed curves $\Arc$, both metrics coincide.) The geodesic equation in the metric \eqref{curvesmetric} is just $\eta_{tt}=0$, while the equation in the invariant metric \eqref{rightinvariantmetric} is the much more complicated\footnote{This is a typical sort of tradeoff for invariance: the same thing happens in fluid mechanics, when we consider the diffeomorphism group $\Diff(M)$ of a Riemannian manifold $M$ and the volumorphism group $\Diffmu(M) = \{ \eta\in \Diff(M): \eta^*\mu = \mu\}$ where $\mu$ is the Riemannian volume form on $M$. The simplest metric on $\Diff(M)$ is the non-invariant metric $\llangle u\circ\eta, v\circ\eta\rrangle_{\eta} = \int_M \langle u, v\rangle\circ\eta \, d\mu$, for which the geodesic equation is $\eta_{tt}=0$ (which leads to Burgers' equation $u_t + \nabla_uu = 0$, where $\eta_t=u\circ\eta$). The right-invariant metric is $\llangle u\circ\eta, v\circ\eta\rrangle_{\eta} = \int_M \langle u,v\rangle \, d\mu$, on which the geodesic equation is $u_t + \nabla_uu + (\diver{u}) u + \tfrac{1}{2} \grad \lvert u\rvert^2 = 0$, again with $\eta_t=u\circ\eta$. Both metrics agree on the submanifold $\Diffmu(M)$, on which the projected geodesic equation is $u_t+\nabla_uu = -\grad p$ with $\diver{u}=0$.}
nonlinear elliptic equation 
$$ \frac{\partial}{\partial t}\Big( \lvert \eta_s\rvert \, \eta_t\Big) + \frac{1}{2} \, \frac{\partial}{\partial s}\Big( \frac{\lvert \eta_t\rvert^2}{\lvert \eta_s\rvert} \, \eta_s\Big) = 0.$$

\subsection{The Michor-Mumford metric on the arc space}

Now we formally identify the quotient $\Imm_2/\Diff$ with the space $\Arc$ in order to compare the induced metric on $\Arc$ to our metric \eqref{curvesmetric}.

\begin{theorem}\label{mmmetricthm}
Define the standard reparametrization map $\Phi\colon \Imm_2\to \Arc$ by $\Phi(\eta) = \eta\circ h^{-1}$ where $h\in \Diff([-1,1])$ is given by $h(s) = \int_0^s \lvert \eta'(x)\rvert \, dx$. Then for any $k\in \Diff([-1,1])$, we have $\Phi(\eta\circ k) = \Phi(\eta)$, so that $\Phi$ is invariant under the action of $\Diff$. Hence $\Phi$ descends to a map from the quotient space $\Imm_2/\Diff$ into $\Arc$. 

There is a unique metric on $\Arc$ defined by the condition that $\Phi$ is a Riemannian submersion, and it is given by 
\begin{equation}\label{mmmetric}
\llangle u, v\rrangle_{\gamma} = \int_{-1}^1 \langle u(s), \gamma'(s)^{\perp}\rangle \langle v(s), \gamma'(s)^{\perp}\rangle \, ds,
\end{equation}
where $u$ and $v$ are in $T_{\gamma}\Arc$ and $\gamma'(s)^{\perp}$ is the rotation of the unit vector $\gamma'(s)$ in $\mathbb{R}^2$ by $90^{\circ}$.
\end{theorem}

\begin{proof}
First, if $\gamma = \Phi(\eta)$ then $\gamma'(s) = \frac{\eta'(h^{-1}(s))}{\lvert \eta'(h^{-1}(s))\rvert}$, so that $\lvert \gamma'(s)\rvert \equiv 1$. Hence $\Phi$ actually maps into $\Arc$. 
Now if $k\colon [-1,1]\to [-1,1]$ is a diffeomorphism with $k'>0$, then for the curve $\eta\circ k$ we get $\tilde{h} = h\circ k$, so that $\tilde{\gamma} = (\eta\circ k)\circ (h\circ k)^{-1} = \eta\circ h^{-1} = \gamma$. Hence $\Phi(\eta\circ k)=\Phi(\eta)$, so that $\Phi$ is invariant under the reparametrization action.

To get the metric, we first compute the derivative $D\Phi$. For any odd vector field $w$ along an odd immersion $\eta$, let $\chi(\varepsilon, s) = \eta(s) + \varepsilon w(s)$; then for sufficiently small $\varepsilon$ the map $s\mapsto \chi(\varepsilon, s)$ is still an odd immersion, and we have $(D\Phi)_{\eta}(w) = \frac{d}{d\varepsilon}\big|_{\varepsilon=0} \Phi\big(\chi(\varepsilon)\big)$. It is then easy to compute that 
\begin{multline}\label{DPhi}
(D\Phi)_{\eta}(w) = \beta\circ h^{-1}, \quad \text{where}\\ 
\beta(s) = w(s) - \big(\textstyle\int_0^s \langle w'(x), \eta'(x)/\lvert \eta'(x)\rvert\rangle  \, dx\big) \eta'(s)/\lvert \eta'(s)\rvert.
\end{multline}
If $\gamma = \Phi(\eta)=\Phi\circ h^{-1}$, then we can check that $z=(D\Phi)_{\eta}(w)$ actually satisfies $\langle z', \gamma'\rangle\equiv 0$ as expected for any $w$, and that the kernel of $(D\Phi)_{\eta}$ is the vertical space 
$$ \mathcal{V}_{\eta} = \left\{ f\eta' : \text{$f\colon [-1,1]\to \mathbb{R}$ is odd}\right\}.$$
The horizontal space is the orthogonal complement of the vertical space in the metric \eqref{rightinvariantmetric}, which is 
$$ \mathcal{H}_{\eta} = \left\{ f(\eta')^{\perp} : \text{$f\colon [-1,1]\to\mathbb{R}$ is odd}\right\}.$$

The metric on $\Arc$ which makes $\Phi$ a submersion is given for $z\in T_{\gamma}\Arc$ by 
$ \llangle z, z\rrangle_{\gamma} = \llangle w, w\rrangle_{\eta}$, where $\eta$ is any curve with $\Phi(\eta)=\gamma$, $w\in T_{\eta}\Imm_2$ is any \emph{horizontal} vector field with $(D\Phi)_{\eta}(w)=z$, and the right side is computed using the invariant metric \eqref{rightinvariantmetric}. Invariance of the metric \eqref{rightinvariantmetric} ensures that we get the same $\llangle z,z\rrangle_{\gamma}$ no matter which $\eta$ we use, so we might as well use $\eta=\gamma$. Then we can compute that the unique horizontal $w$ with $(D\Phi)_{\gamma}(w) = z$ is $w(s) = \langle z(s), \gamma'(s)^{\perp}\rangle \gamma'(s)^{\perp}$, and formula \eqref{mmmetric} follows. 
\end{proof}

Next let us compute the geodesic equation on $\Arc$ in the Michor-Mumford metric. The following lemma is helpful in finding compatibility conditions for it.

\begin{lemma}\label{compatlemma}
If $\eta\colon [0,T]\times [-1,1]\to \mathbb{R}^2$ is a smooth curve, and if we define $\ell = \lvert \eta_s\rvert$ and
\begin{equation}\label{generalabkappadef}
\kappa = \frac{\langle \eta_{ss}, \eta_s^{\perp}\rangle}{\ell^3}, 
\quad \omega = \frac{\langle \eta_{st}, \eta_s^{\perp}\rangle}{\ell^2}, \quad a=\frac{\langle \eta_t, \eta_s^{\perp}\rangle}{\ell}, \quad b = \frac{\langle \eta_t, \eta_s\rangle}{\ell},
\end{equation}
then we have the compatibility equations
\begin{equation}\label{compatibility}
\ell_t = b_s - a\kappa \ell, \quad 
\partial_t(\ell \kappa) = \omega_s, \quad \text{and}\quad 
a_s = \ell \omega - b\kappa \ell
\end{equation}
\end{lemma}

\begin{proof}
We write \begin{equation*}
\eta_t  = \frac{a}{\ell} \eta_s^{\perp} + \frac{b}{\ell} \eta_s,
\qquad \eta_{st} = \omega \eta_s^{\perp} + \frac{\ell_t}{\ell} \eta_s,\qquad
\eta_{ss} = \kappa \ell \eta_s^{\perp} + \frac{\ell_s}{\ell} \eta_s.
\end{equation*}
Differentiating $\eta_{t}$ with respect to $s$ and matching coefficients with $\eta_{st}$, we get $a_s/\ell = \omega - b\kappa$ and $\ell_t = b_s - a\kappa \ell$. Then using $\eta_{sts} = \eta_{sst}$ we obtain $\partial_t(\kappa \ell) = \omega_s$.
\end{proof}

\begin{theorem}\label{mmgeodesicthm}
A geodesic $\eta$ in $\Arc$ with the metric \eqref{mmmetric} satisfies the equations 
\begin{equation}\label{mmgeodesic}
a_t = \tfrac{1}{2} \kappa a^2 + ba_s, \qquad \kappa_t = \omega_s, \qquad b_s = \kappa a, \qquad a_s = \omega - \kappa b,
\end{equation}
where
\begin{equation}\label{termsdef}
a = \langle \eta_t, \eta_s^{\perp}\rangle, \quad b = \langle \eta_t, \eta_s\rangle, \quad 
\kappa = \langle \eta_{ss}, \eta_s^{\perp}\rangle, \quad \omega = \langle \eta_{st}, \eta_s^{\perp}\rangle,
\end{equation}
and $a(-1)=a(1)=0$. 
\end{theorem}

\begin{proof}
If $\eta$ is a curve in $\Arc$ and $u$ is a  variation field along $\eta$, then it is easy to compute that the first variation of energy in the direction $u$ is 
\begin{equation}\label{evariation} 
\int_0^T \big(\langle u^{\perp}, a\eta_t\rangle\big)_{s=-1}^{s=1} \, dt 
+ \int_0^T \int_{-1}^1 \langle u^{\perp}, a_t \eta_s - a_s \eta_t\rangle \, ds\,dt.
\end{equation}
For $\eta$ to be a geodesic, this must vanish for every $u\in T_{\eta}\Arc$, i.e., whenever $\langle u, \eta'\rangle\equiv 0$. From the first term we get the boundary condition $a(-1)=a(1)=0$, and from the second term we get the equation $a_t = ba_s + \tfrac{1}{2}\kappa a^2$. 
%
The other equations are
obtained by setting $\ell\equiv 1$ in Lemma \ref{compatlemma}.
\end{proof}

The geodesic equations take a slightly different form than that given in \cite{MM}; there the authors use
the normalization $b\equiv 0$ rather than our normalization $\ell\equiv 1$. Of course, the images of the curves in $\mathbb{R}^2$ are necessarily the same.

The drawback of the $L^2$ Michor-Mumford metric, as discussed in \cite{MM}, is that the induced Riemannian distance between elements of $\Arc$ is zero; that is, for any pair of curves $\gamma_1$ and $\gamma_2$ in $\Arc$ and any $\varepsilon>0$, there is a curve $\eta$ in $\Arc$ with $\eta(0)=\gamma_1$ and $\eta(1)=\gamma_2$ such that $\int_0^1 \lVert \dot{\eta}(t)\rVert \, dt < \varepsilon$. As shown in 
Proposition \ref{nondegeneratedistance}, our metric on $\Arc$ \emph{does} give a genuine nondegenerate distance.

We now ask what a right-invariant metric on $\Imm_2([-1,1], \mathbb{R}^2)$ would have to look like in order to give our metric \eqref{curvesmetric} on $\Arc$ as a Riemannian submersion using the procedure in Theorem \ref{mmmetricthm}. 

\begin{theorem}\label{modifiedsubmersion}
Let us define a Riemannian metric on $\Imm_2([-1,1], \mathbb{R}^2)$ as follows: if $\eta$ is a curve and $w$ is a vector field along $\eta$, let
\begin{equation}\label{modifiedinvariant}
\llangle w, w\rrangle_{\eta} = \int_{-1}^1 \frac{\langle w(s), \eta'(s)^{\perp}\rangle^2}{\lvert \eta'(s)\rvert} \, ds + \int_{-1}^1 \lvert \eta'(s)\rvert \left( \int_0^s \frac{\langle w'(x), \eta'(x)\rangle}{\lvert \eta'(x)\rvert} \, dx\right)^2 \, ds.
\end{equation}
Then the metric \eqref{modifiedinvariant} is invariant under the reparametrization action by $\Diff([-1,1])$, and the map $\Phi\colon \Imm_2\to\Arc$ defined as in Theorem \ref{mmmetricthm} is a Riemannian submersion onto the arc space $\Arc$ in the metric \eqref{curvesmetric}. 
\end{theorem}

\begin{proof}
To check invariance, we just need to verify $\llangle w\circ k, w\circ k\rrangle_{\eta\circ k} = \llangle w, w\rrangle_{\eta}$ for any increasing diffeomorphism $k$ of $[-1,1]$. This is straightforward from the change of variables formula.

To check the submersion condition, we suppose we have a curve $\gamma$ with $\lvert \gamma'(s)\rvert = 1$ and that $w$ is a horizontal vector field along $\gamma$, i.e., that $\langle w(s), \gamma'(s)\rangle \equiv 0$. Then as in Theorem \ref{mmmetricthm}, we have 
$$(D\Phi)_{\gamma}(w)(s) = w(s) - \big( \textstyle \int_0^s \langle w'(x), \gamma'(x)\rangle \, dx\big) \gamma'(s).$$ Suppose $z\in T_{\gamma}\Arc$, i.e., that $z$ is a vector field along $\gamma$ with $\langle z'(s), \gamma'(s)\rangle \equiv 0$. Then $z(s) = f(s) \gamma'(s) + g(s)\gamma'(s)^{\perp}$ where $f'(s) = \kappa(s) g(s)$, and to get $(D\Phi)(w) = z$ where $w$ is horizontal, we must have $w(s) = g(s) \gamma'(s)^{\perp}$. We can then check that the definition \eqref{modifiedinvariant} yields 
$$ \llangle w,w\rrangle_{\gamma} = \int_{-1}^1 \big[ f(s)^2 +g(s)^2 \big] \, ds, $$
as desired.
\end{proof}

Of course there are other choices for \eqref{modifiedinvariant}; only the inner product of horizontal vectors is determined by the submersion condition, and we can use any formula at all for vertical vectors.

\subsection{The $\dot{H}^1$ metric on $\Arc$}

Finally we relate both our metric \eqref{curvesmetric} and the Michor-Mumford metric \eqref{mmmetric} to another choice of distance on the arc space $\Arc$. Klassen et al.~\cite{KSMJ} pointed out that unit-speed curves in $\mathbb{R}^2$ can most easily be represented in terms of their angular representation $\theta$ defined by $\gamma'(s) = \big( \cos{\theta(s)}, \sin{\theta(s)}\big)$, as we did in Theorem \ref{arcsubmfd} to get a coordinate chart on $\Arc^m$. Since $\gamma(0)=0$ in our space, we obtain $\gamma$ by integrating: 
\begin{equation}\label{thetatogamma}
\gamma(s) = \left( \int_0^s \cos{\theta(x)}\,dx, \int_0^s \sin{\theta(x)}\,dx\right).
\end{equation}
Since the space of (even) functions $\theta \colon [-1,1]\to \mathbb{R}$ is a linear space, it has a simple choice of Riemannian metric arising from the standard Hilbert structure. That is, if $\gamma$ is a curve with angular representation $\theta$, and $\omega$ is a vector field along $\theta$, then the KSMJ metric is
\begin{equation}\label{ksmjmetric}
\llangle \omega, \omega\rrangle_{\theta} = \int_{-1}^1 \omega(s)^2 \, ds.
\end{equation}
In the physical space $\mathbb{R}^2$, the angular tangent vector $\omega$ corresponds to the vector field 
$$ 
u(s) = \int_0^s \omega(x) \gamma'(x)^{\perp}\,dx,$$ and thus the KSMJ metric comes from the $\dot{H}^1$ metric on $\Arc$ given by 
\begin{equation}\label{ksmjdot}
\llangle u, u\rrangle_{\gamma} = \int_{-1}^1 \langle u'(s), u'(s)\rangle \, ds.
\end{equation}
Again we note that KSMJ were interested in periodic curves on $S^1$ rather than odd curves on $[-1,1]$, but the formulas are generally quite similar apart from normalizations. The metric \eqref{ksmjdot} has also been studied by Younes et al.~\cite{younes}; a similar metric arises in the study of the Hunter-Saxton equation as well (see Khesin-Misio{\l}ek~\cite{KM} and Lenells~\cite{lenells}). 
%
%
%
%

In a sense then, our metric \eqref{curvesmetric} lies between the metric \eqref{ksmjdot} (for which there are unique minimizing geodesics and a nondegenerate distance) and the metric \eqref{mmmetric} (for which geodesics cannot be minimizing and the distance is always degenerate). Our geodesics fail to be minimizing by Theorem \ref{notc1}, but our distance is nondegenerate by Proposition \ref{nondegeneratedistance}. Furthermore the geometry induced can be approximated by finite-dimensional objects, as in \cite{whipsandchains}, where unit-speed curves are well-approximated by a chain of points joined by rigid rods of fixed length, which may be helpful for numerical approximations of curves in this geometry.

\appendix

\section{Other boundary conditions}\label{circleappendix}

In this paper we have exclusively studied the boundary condition corresponding to a whip
with one fixed end and one free end. This is the most physically relevant condition for an 
actual whip (a person swings the whip to give it an initial position and velocity, then holds the handle
basically fixed while the other end swings freely). As shown in \cite{whipsandchains}, the easiest 
way to handle the technical complications of a fixed end is to extend the curve through the origin to
be odd; then the boundary conditions work out automatically. Hence we have essentially reduced the 
situation with one fixed and one free end to the situation with two free ends. There is no substantial 
difference in any of the results when dealing with two free ends even if the curve is not odd.
With two fixed ends, the situation is more complicated. (Physically this might represent a jump rope being held at both ends.) The same technical issues arise, but now it is less obvious how to extend the whip to be odd on both ends; of course it can be done, but then we end up with an infinite string and lose some of the benefits of compactness.

Geometrically, however, none of these boundary conditions are nearly as relevant as the periodic condition, since we are interested in curves that form boundaries of planar objects (and hence cannot themselves have a boundary). Many aspects of this situation are technically easier than the one-fixed/one-free condition we have considered, since we can do everything in terms of ordinary Sobolev spaces on the circle rather than weighted Sobolev spaces on the interval. The major differences are in the upper and lower bounds of the tension, and in the fact that periodicity forces $\int_{S^1} \gamma'(s) \, ds=0$, which shows up as an extra constraint in some equations.

Throughout this appendix we work with the circle of length $1$. The space of curves is the ordinary Sobolev space $\Curve^m(S^1) = H^m(S^1, \mathbb{R}^2)$, and the subset of arc-length parametrized curves is $\Arc^m(S^1) = \{\gamma \in \Curve^m : \lvert \gamma'(s)\rvert^2 \equiv 1\}$. Let $\Func^{m-1}(S^1)$ denote the space of real-valued functions of class $H^{m-1}$. The proof that $\Arc^m(S^1)$ is a submanifold of $\Curve^m(S^1)$ is both simpler and works in more cases for periodic boundary conditions than for one fixed and one free end.

\begin{theorem}\label{circsubmfd} 
If $m\ge 2$, then $\Arc^m(S^1)$ is a smooth Hilbert submanifold of $\Curve^m(S^1)$.
\end{theorem}

\begin{proof}
Define $\mathcal{J}\colon \Curve^m \to \Func^{m-1}$ by the formula $\mathcal{J}(\eta)(s) = \lvert \eta'(s)\rvert^2$. Since $H^{m-1}$ functions are closed under multiplication for $m\ge 2$, it is easy to see that $\mathcal{J}$ is well-defined and $C^{\infty}$. The differential is easy to compute: we have 
$ (D\mathcal{J})_{\eta}(u) = 2\langle u', \eta'\rangle.$ Hence for any $\gamma\in \Arc^m(S^1)$ we can prove it is surjective; let $f\in \Func^{m-1}$ be any real-valued function. Write $\gamma'(s) = \phi(s)\partial_x + \psi(s) \partial_y$ for $H^{m-1}$ functions $\phi$ and $\psi$ satisfying $\phi^2+\psi^2\equiv 1$, and define $u$ so that $u' = (f\phi - g\psi) \,\partial_x + (f\psi + g\phi) \,\partial_y$, where $g=a\phi + b\psi$ and $a$ and $b$ are constants chosen so that $\int_0^1 u'(s) \, ds = 0$ (which is necessary for $u$ to be periodic). It is easy to see that such constants can always be chosen as long as $\phi$ and $\psi$ are not constant, and the only way that could happen is if $\gamma$ were a geodesic. But there are no closed geodesics in $\mathbb{R}^2$.\footnote{The same construction works for the arc space on any Riemannian manifold $M$, and $\Arc^m(S^1) = \{\gamma\colon S^1\to M : \lvert \gamma'\rvert \equiv 1\}$ fails to be a manifold if $M$ contains any closed geodesic of length $1$.} Clearly $u$ constructed this way is in $H^m$, and so $D\mathcal{J}$ is surjective at any $\gamma\in \Arc^m(S^1)$. Hence if $\mathbb{1}$ is the constant function $1$ on $S^1$, then $\mathcal{J}^{-1}(\mathbb{1})=\Arc^m(S^1)$ is a Hilbert submanifold.
\end{proof}

The tangent space $T_{\gamma}\Arc^m$ still consists of $H^m$ vector fields $u$  with $\langle u', \gamma'\rangle\equiv 0$, so that the (formal) orthogonal complement is still $$(T_{\gamma}\Arc^m)^{\perp} = \left\{ \frac{d}{ds}(\sigma \gamma') : \sigma\in \Func^{m+1}\right\}.$$ (Note that as in Proposition \ref{orthoprojprop}, this only makes sense in $T_{\gamma}\Curve^m$ if $\gamma\in \Arc^{m+2}$.) So the orthogonal projection is 
$P_{\gamma}(z) = z - (\sigma \gamma')',$ where 
\begin{equation}\label{sigmacircle}
\sigma'' - \lvert \gamma''\rvert^2 \sigma = \langle z', \gamma'\rangle, \qquad \sigma(0)=\sigma(1), \; \sigma'(0)=\sigma'(1).
\end{equation}
To prove equation \eqref{sigmacircle} always has a solution, we construct the Green function for it.

\begin{proposition}\label{greencircle}
Suppose $\gamma$ and $z$ are smooth. If $\lvert \gamma''\rvert$ is not identically zero, then equation \eqref{sigmacircle} has a unique solution $\sigma$.
\end{proposition}

\begin{proof}
Let $G\colon S^1\times S^1\to \mathbb{R}$ denote the Green function, satisfying 
$$ G_{ss}(s,x) - \lvert \gamma''(s)\rvert^2 G(s,x) = -\delta(s-x), \quad G(0,x)=G(1,x), G_s(0,x)=G_s(1,x).$$
If we can find the Green function, then the solution of \eqref{sigmacircle} is given by 
$ \sigma(s) = -\int_0^1 G(s,x) \langle z'(x), \gamma'(x)\rangle \, dx$. 

Translating by $x$, we easily see that $G(s,x) = \varphi(s)$ where 
\begin{equation}\label{varphi}
\varphi''(s) = \kappa(s)^2 \varphi(s), \qquad \varphi(0)=\varphi(1), \quad \varphi'(0)+1 = \varphi'(1),
\end{equation}
with $\kappa(s) = \lvert \gamma''(s+x)\rvert$. So we just need to prove that the boundary value problem \eqref{varphi} has a solution if $\kappa$ is not identically zero. 

Let $\varphi_1$ and $\varphi_2$ denote the solutions of \eqref{varphi} with boundary conditions $\varphi_1(0)=1$, $\varphi_1'(0)=0$, $\varphi_2(0)=0$, and $\varphi_2'(0)=1$. Clearly $\varphi_1(s)\ge 1$ and $\varphi_2'(s)\ge 1$ for all $s$. We can write $\varphi = A\varphi_1 + B\varphi_2$ where $A$ and $B$ 
satisfy 
\begin{equation}\label{ABdet}
A[ \sigma_1(1)-1] + B\sigma_2(1) = 0 \quad \text{and} \quad A\sigma_1'(1) + B[\sigma_2'(1)-1] = 1.
\end{equation}
Using the reduction of order trick, we can compute that $\sigma_2(s) = \sigma_1(s) \int_0^s dr/\sigma_1(r)^2$, so that the determinant of the system \eqref{ABdet} is 
$$ -\frac{1}{\sigma_1(1)} [\sigma_1(1)-1]^2 - \sigma_1'(1) \int_0^1 dx/\sigma_1(x)^2,$$
which is always negative unless $\sigma_1(1)=1$ and $\sigma_1'(1)=0$ (which happens if and only if $\kappa\equiv 0$). Hence we can solve for $A$ and $B$, and so we obtain the Green function.

Uniqueness is trivial using a standard energy argument.
\end{proof}

Now we obtain upper and lower bounds for $\sigma$. 

\begin{proposition}\label{greencirclebounds}
The solution $\varphi$ of \eqref{varphi} satisfies
\begin{equation}
e^{-\varrho/2}/\varrho \le \varphi(s) \le 1 + \frac{1}{4\pi^2}, \quad \text{where} \quad \varrho = \int_0^1 \kappa^2(x)\,dx,
\end{equation}
for all $s\in S^1$. 
\end{proposition}

\begin{proof}
To prove the lower bound, set $\gamma=\ln{\varphi}$, so that $\gamma''=\kappa^2-\gamma'^2$. Let $u$ be a point where $\gamma'=0$; since $\gamma$ is convex, $\gamma(u)$ is the minimum of $\gamma$. Integrating by parts twice from $0$ to $u$ we get 
$ \gamma(u)\ge \gamma(0) - \int_0^1 x \kappa^2(x)\,dx,$
and similarly integrating from $u$ to $1$ we get $\gamma(u)\ge \gamma(1) - \int_0^1 (1-x) \kappa^2(x)\,dx$. Since $\gamma(0)=\gamma(1)$, averaging these estimates gives 
$$
\gamma(u) \ge \gamma(0) - \tfrac{1}{2} \int_0^1 \kappa^2(x) \, dx.
$$
Thus we get $\varphi(u) \ge \varphi(0) e^{-\varrho/2}$. 

Now we need a lower bound for $\varphi(0)$. Since $\varphi'(1)=\varphi'(0)+1$ and $\varphi(0)=\varphi(1)$, we have $\gamma'(1)=\gamma'(0)+1/\varphi(0)$, so that 
$$ \frac{1}{\varphi(0)} = \gamma'(1)-\gamma'(0) = \int_0^1 \kappa^2(s) \, ds - \int_0^1 \gamma'(s)^2 \, ds \le \varrho.$$
Combining this with our estimate for $\varphi(u)$ gives the lower bound $\varphi(u) \ge e^{-\varrho/2}/\varrho$. 

The upper bound is obtained differently. Writing $\varphi=A\varphi_1+B\varphi_2$ as in Proposition \ref{greencircle}, we get 
$$\varphi(0)=A = \frac{L\varphi_1(1)}{L\varphi'(1) + [\varphi_1(1)-1/\varphi_1(1)]^2},$$
where $L=\int_0^1 dx/\varphi_1(x)^2$. Hence $A\le \varphi_1(1)/\varphi_1'(1)$. 
Now since $\varphi_1$ is convex, we know $\varphi_1'(1)\ge \varphi_1(1)-1$, which implies $A\le 1+1/\varphi'(1)$. 
Finally, we use the fact that $\varphi_1(s)\ge 1$ for all $s$ to get 
$$ \varphi_1'(1)= \int_0^1 \varphi_1(s) \kappa^2(s) \, ds \ge \int_0^1 \kappa(s)^2 \, ds \ge \left( \int_0^1 \lvert \gamma''(s+x)\rvert \, ds\right)^2 \ge 4\pi^2,$$
using the well-known bound on the total curvature of a closed curve in the plane (see for example
do Carmo~\cite{docarmo}, Section 5-7).
\end{proof}

As a corollary, we get a bound for the curvature in exactly the same way as Theorem \ref{curvaturethm}.

\begin{proposition}\label{circlecurvature}
The sectional curvature of $\Arc(S^1)$ in the $L^2$ metric is bounded below at $\gamma\in \Arc(S^1)$ 
by $ 4\pi^2 e^{-\varrho/2}/\varrho$, where $\varrho=\int_{S^1} \kappa(s)^2 \, ds$, and is unbounded above.
\end{proposition}

\begin{proof}
The proof is the same as Theorem \ref{curvaturethm}. The factor $4\pi^2$ comes from the fact that the smallest eigenvalue of the derivative operator has size $2\pi$.
\end{proof}

We can prove using techniques similar to \cite{whipsandchains} the analogue of Theorem \ref{maincompanion}: that solutions of the geodesic equation exist as long as $\eta(0)\in \Arc^3(S^1)$ and $\dot{\eta}(0)\in T_{\gamma}\Arc^2(S^1)$; the weakening comes from the fact that we never need to use the estimates of Lemma \ref{weightedestimateslemma} to fix the weighting at the endpoints.

Unboundedness of the curvature again implies that the exponential map cannot be $C^1$; this can be checked explicitly using simple explicit solutions. The simplest explicit solution of the geodesic equation on $S^1$ is 
\begin{equation}\label{explicitcirclegeodesic}
\eta(t,s) = \gamma(s+\omega t), \qquad \omega(t,s) = \omega^2,
\end{equation}
where $\gamma\colon S^1\to\mathbb{R}^2$ is any closed curve. We can compute Jacobi fields along such a curve explicitly (for example, if $\gamma$ is a circle) and show that they have zeroes for arbitrarily short times, as in Theorem \ref{notc1}.

One might object to the notion that unphysical solutions such as \eqref{explicitcirclegeodesic} should be allowed, especially in application to shape recognition, since the image of the curve doesn't change with time. The typical way to resolve this (see for example \cite{younes}) is to quotient out by the translations. The metric on $\Arc(S^1)$ is obviously invariant under the action by translations, so we get a metric on the quotient $\Arc(S^1)/S^1$. Since \eqref{explicitcirclegeodesic} is always a geodesic, this is a Riemannian submersion with totally geodesic fibers. This corresponds to requiring tangent vectors to satisfy not only $\langle u', \gamma'\rangle \equiv 0$ but also $\int_{S^1} \langle u(s), \gamma'(s)\rangle \, ds = 0$. The new orthogonal space $T_{\gamma}(\Arc(S^1)/S^1)$ then consists of fields of the form $\frac{d}{ds} (\sigma \gamma') + c \gamma'$ for functions $\sigma$ and constants $c$. The geodesic equation is then 
$$ \eta_{tt} = \partial_s(\sigma \eta_s) + c\eta_s, \qquad \sigma_{ss}-\lvert \eta_{ss}\rvert^2 \sigma = -\lvert \eta_{st}\rvert^2,$$
and we have $\frac{d}{dt} \int_{S^1} \langle \eta_t, \eta_s \rangle \, ds = c$, so that $c=0$ in order to preserve horizontality. Hence the same geodesic equation guarantees that the translations disappear as long as $\int_{S^1} \langle \eta_t, \eta_s\rangle \, ds=0$ initially.

\section{Removing the length constraint}\label{nolengthappendix}

We note that one drawback to these equations in shape analysis is that the space of all odd curves (modulo reparametrizations) is not exactly the same as the arc-length parametrized curves, since the space of all curves includes those of arbitrary length while ours consists only of curves of length $2$. To extend this, we would have to work with a slightly different version of $\Arc$. Let 
$$\overline{\Arc} = \left\{\gamma\colon [-1,1]\to \mathbb{R}^2 : \lvert \gamma'\rvert \equiv \text{const} \text{ and } \gamma(-s)=-\gamma(s) \, \forall s\right\}.$$ 

\begin{theorem}
The geodesic equation on $\overline{\Arc}$ in the weak metric \eqref{curvesmetric} is given by 
\begin{align}
\eta_{tt} &= \partial_s(\sigma \eta_s), \qquad \lvert \eta_s\rvert^2 = \ell^2, \label{stringevolloose} \\
\ell^2 \sigma_{ss} - \lvert \eta_{ss}\rvert^2 + \lvert \eta_{st}\rvert^2 &= C, \qquad \int_{-1}^1 \sigma \, ds = 0,\qquad \sigma(-1)=\sigma(1)=0. \label{stringconstraintloose}
\end{align}
Equation \eqref{stringconstraintloose} always has a solution for any given $\eta$. 
Here $\ell$ and $C$ are constant in space but not necessarily in time, and we have $\frac{d^2}{dt^2}(\ell^2) = 2C$. 
\end{theorem}

\begin{proof}
The condition for a vector field $u$ along $\gamma\in\overline{\Arc}$ to be in $T_{\gamma}\overline{\Arc}$ is that $\frac{d}{ds} \langle u', \gamma'\rangle \equiv 0$, so that the orthogonal complement consists of fields of the form $\frac{d}{ds} (\sigma \gamma')$ for functions $\sigma$ with $\int_{-1}^1 \sigma(s) \, ds = 0$. The geodesic equation therefore still has the form of the wave equation \eqref{stringevolloose}. Here $\lvert \eta_s\rvert^2 = \ell^2$ for some $\ell$ which is constant in space but may depend on time. 

To find the equation for $\sigma$ we differentiate the constraint $\frac{1}{2}\partial_s \lvert \eta_s\rvert^2 = 0$ twice with respect to time to get $\partial_s( \langle \eta_{stt}, \eta_s\rangle + \lvert \eta_{st}\rvert^2) = 0$. Plugging in $\eta_{tt}$ from \eqref{stringevolloose}, we obtain 
$$ \partial_s\big( \ell^2 \sigma_{ss} - \lvert \eta_{ss}\rvert^2 + \lvert \eta_{st}\rvert^2\big) = 0,$$
which integrates to \eqref{stringconstraintloose}.

To prove we can actually solve \eqref{stringconstraintloose}, we consider the slightly modified Green function $G_{\ell}(s, x)$ satisfying 
$$\ell^2 \partial_s^2 G(s,x) - \lvert \eta_{ss}(s)\rvert^2 = -\delta(s-x), \qquad G(-1,x)=G(1,x)=0.$$
If we write the solution $\sigma$ in terms of this Green function, then $C$ is determined by the fact that the integral of $\sigma$ vanishes, i.e., that
 $$ \int_{-1}^1 \int_{-1}^1 G(s,x) \big( \lvert \eta_{tx}(x)\rvert^2 - C\big) \, dx \, ds = 0.$$ 
By Theorem \ref{greenthm}, $G(s,x)\ge 0$ for all $s$ and $x$ in the square $[-1,1]^2$ and is zero only on the boundary, which means we can always solve this equation for $C$.
It is easy to compute that $\frac{d^2}{dt^2} (\ell^2) = 2 C$.
\end{proof}

Although we can solve \eqref{stringconstraintloose}, it is far from clear that we can solve \eqref{stringevolloose}: the difficulty is that since $\sigma$ integrates to zero, it cannot be strictly positive, which 
means \eqref{stringevolloose} is always of mixed type, and hence substantially more difficult to analyze. We also lose positivity of the sectional curvature as in Theorem \ref{curvaturethm}. In addition, the fact that $C$ is always positive except in degenerate cases means that $\ell$ is always increasing, so that in particular $\Arc$ is not a totally geodesic submanifold of $\overline{\Arc}$. We leave these complicated issues aside for now however, noting merely that the situation is not at all improved by changing the boundary conditions as in Appendix \ref{circleappendix}.

\makeatletter \renewcommand{\@biblabel}[1]{\hfill#1.}\makeatother

\end{document}